\documentclass{amsart}
\usepackage{cite}
\usepackage[all]{xy}
\usepackage{graphicx}
\usepackage{amssymb}
\usepackage{verbatim}
\usepackage[numbers,sort&compress]{natbib} 
\usepackage[bookmarksnumbered, bookmarksopen,
colorlinks,citecolor=blue,linkcolor=blue]{hyperref}

\newtheorem{Theorem}{Theorem}[section]
\newtheorem{Lemma}[Theorem]{Lemma}
\theoremstyle{Definition}

\newtheorem{Example}[Theorem]{Example}

\newtheorem{Corollary}[Theorem]{Corollary}

\theoremstyle{Remark}
\newtheorem{Remark}[Theorem]{Remark}

\theoremstyle{notation}

\numberwithin{equation}{section}

 \def\3{\operatorname{3}}

\begin{document}

\title{Formality on the rationalizations of simply connected CW complexes}

\author{Jingwen Gao}
\address{School of Mathematical Sciences,
         Nankai University,
         Tianjin 300071, P.R.China}

\email{2120190022@mail.nankai.edu.cn}
\thanks{The authors were supported in part by Tianjin Natural Science Foundation (Grant No. 19JCYBJC30200).}

\author{Xiugui Liu}
\address{School of Mathematical Sciences and LPMC,
         Nankai University,
         Tianjin 300071, P.R.China}

\email{xgliu@nankai.edu.cn}


\date{\today}



\begin{abstract}
In this paper, we show that for a simply connected CW complex $Y$ with finite dimensional rational cohomology, if $H^{*}(Y;\mathbb{Q})$ vanishes in odd degrees and satisfies a specific condition,  then the rationalization of $Y$, $Y_\mathbb{Q}$, is formal.
\end{abstract}

\maketitle

\section{Introduction}
We assume that all topological spaces in this paper are path connected.
For a commutative cochain algebra $(A,d)$ satisfying $H^0(A)=\mathbb{K}$ where $\mathbb{K}$ is a field of characteristic $0$, $(A,d)$ is called formal if it is weakly equivalent to the cochain algebra $(H(A;\mathbb{K}), 0)$. A path connected topological space, $X$, is formal if $A_{PL}(X; \mathbb{Q})$ is a formal cochain algebra.

Formality is an important property for both topological spaces  and  cochain algebras. For example, we use formality of K$\ddot{{\rm a}}$hler manifolds to distinguish non simply connected symplectic manifolds with no K$\ddot{{{\rm a}}}$hler structure \rm{\cite{benson, lupton}}. But there  exist non formal simply connected spaces, I. K. Babenko and I. A. Taimanov \rm{\cite{ik}} have constructed infinitely many non formal simply connected manifolds.
The aims of this paper is to find a class of spaces whose rationalizations are formal.

Let $Y$ be a simply connected CW complex with finite dimensional rational cohomology.
Let $V\subset H^{+}(Y;\mathbb{Q})$ be a graded vector space such that
$$H^+(Y;\mathbb{Q})=H^+(Y;\mathbb{Q})\cdot{H^+(Y;\mathbb{Q})}\oplus{V}.$$  The inclusion $V\hookrightarrow {H^+(Y;\mathbb{Q})}$ extends to a morphism $\varphi:(\Lambda{V},0)\rightarrow({H^*(Y;\mathbb{Q}),0)}$ of cochain algebras. Let $V^{\prime}=\{v_j|~j\in A\}$ be the set of generators of $V$, where $A$ is a totally ordered set.
The following is our main result of this paper.
\begin{Theorem}\label{1.1} With notations as above.
	Let $E=\{\varphi(v_{t_1}v_{t_2}\cdots v_{t_k})\neq{0}|~v_
	{t_j}\in V^{\prime},~t_j\in A,~t_1\leq t_2\leq\cdots\leq t_k,~k\geq 2,~k\in{\mathbb{Z}}\}$.
	If $H^*(Y;\mathbb{Q})$ vanishes in odd degrees and  satisfies one of the following conditions:
	\begin{itemize}
		\item[(i)]
		$ E=\emptyset$, that is, the cup product is trivial;
		\item[(ii)]
		$E$ is linearly independent,
	\end{itemize}
 then the rationalization of $Y$,
	$Y_\mathbb{Q}$, is formal.
	
\end{Theorem}

As a special case of Theorem \ref{1.1}, we can obtain directly the following corollary.
\begin{Corollary}\label{1.2}
	With notations as above.
	Let $F=\{n_j\in\mathbb{Z}|~0<n_1<\cdots<n_j<\cdots,~ H^{n_j}(Y;\mathbb{Q})\neq 0\}$.
	If $H^{*}(Y;\mathbb{Q})$ vanishes in odd degrees and satisfies one of the following conditions:
	\begin{itemize}
		\item [(i)] F=$\emptyset$;
		\item [(ii)]
		  for any $n_k\in F$,
		$n_k\neq\sum\limits_{1\leq i<k}a_in_i,~a_i\in\mathbb{Z},$
	\end{itemize}
	then the rationalization of $Y$, $Y_\mathbb{Q}$, is formal.
\end{Corollary}

The paper is organized as follows. In Section \ref{22}, we recall some basic knowledge in rational homotopy theory. In Section \ref{33}, we prove Theorem \ref{1.1}.  In Section \ref{44}, an illustrative example of  formality of the  spatial realization of a Sullivan algebra is presented.

\section{Preliminaries}\label{22}

We assume that the reader is familiar with the basics of rational homotopy theory. Our general reference for this material is \cite{06bumu}.

Recall that a singular $n$-simplex in a space $X$ is a continuous map $$\sigma: \Delta^n\rightarrow X,$$
where $n\geq 0$ and $\Delta^n$ is the standard simplex which consists of the points $x=(t_0, \cdots, t_n)$ in $\mathbb{R}^{n+1}$ such that each $t_i\geq 0$ and $ \sum\limits_{i=0}^nt_i=1$. In particular, $e_i=(0,\cdots,t_i=1,\cdots,0)\in \Delta^n$.
If $X\subset \mathbb{R}^m$ is a convex subset then any sequence $x_0,\cdots,x_n\in{X}$ determines the linear simplex $$\langle x_0,\cdots,x_n\rangle: \Delta^n\rightarrow{X},~ \sum t_ie_i\longmapsto{\sum t_ix_i}.$$
There are  important linear simplices
$$\lambda_i=\langle e_0,\cdots,\hat{e}_i,\cdots,e_n\rangle: \Delta^{n-1}\rightarrow{\Delta^{n}},~ 0\leq{i}\leq{n},$$
where $\hat{e}_i$ means omitting.

Denote the cellular chain complex for a CW complex $(X, x_0)$ by $(C_{\ast}(X),\partial_*)$, and let $S_*(X)$ denote the singular chain complex of $X$. Then we have

\begin{Theorem}\label{2.1}\rm{\cite[Theorem 4.18]{06bumu}}
Every CW complex $(X,x_0)$ has a cellular chain model $$m: (C_{\ast}(X),\partial_*)\rightarrow{S_{\ast}(X)},$$  and $m$ is
 a quasi-isomorphism.

\end{Theorem}

\begin{Theorem}\label{2.2}\rm{\cite[Theorem 9.11]{06bumu}} Every simply connected space $Y$ is rationally modelled by a CW complex $X$ for which the differential in the integral celluar chain complex is identically zero.
\end{Theorem}

\begin{Theorem}\label{2.3} \rm{\cite[Propositon 9.8]{06bumu}} Simply connected spaces $X$ and $Y$ have the same rational homotopy type if and only if there is a chain of rational homotopy equivalences
	$$X\leftarrow{\cdots}\rightarrow{\cdots}\leftarrow{\cdots}\rightarrow{Y}.$$
	In particular, if $X$ and $Y$ are CW complexes, then they have the same rational homotopy type if and only if $X_{\mathbb{Q}}\simeq{Y_{\mathbb{Q}}}$.

\end{Theorem}

Recall here that for a simplicial set $K$, the cochain algebra $C^*(K)$ is defined as follows \rm{\cite[Chapter 10(d)]{06bumu}},

(1) $C^*(K)=\{C^p(K)\}_{p\geq{0}}$;

(2) $C^p(K)$ consists of the set maps $K_p\rightarrow \mathbb{K}$ vanishing on degenerate simplices;

(3) For $f\in{C^p(K)},~h\in{C^q(K)}$, the product is given by
$$(f\cup h)(\tau)=(-1)^{pq}f(\partial_{p+1}\cdots\partial_{p+q}\tau)\cdot h(\partial_{0}\cdots\partial_{0}\tau),~\tau\in{K_{p+q}}$$

(4) The differential $d$ is given by
$$d(f)(\tau)=\sum\limits_{i=0}^{p+1}(-1)^{p+i+1}f(\partial_i\tau), ~f\in C^p(K), ~\tau\in K_{p+1}.$$

\begin{Remark}
If $K$ is a simplicial vector space, set maps in (2) mean linear maps. For $k\in K_p$, we set $$K_p=\mathbb{Q}k\oplus B.$$
Denote by
$k^*\in C^p(K)$  a linear map $$k^*:K_p\rightarrow \mathbb{K}$$ such that $k^*(k)=1$
and $k^*(s)=0$ if $s\in B$.
In this paper, there are two important cases we concern: $K=S_*(X)$ and $K=(C_*(X), \partial_*)$.

\end{Remark}
\begin{Theorem}\label{2.5}\rm{\cite[Theorem 10.9]{06bumu}}
	Let $K$ be a simplicial set. Then

(1) There is a natural isomorphism of cochain algebras $C_{PL}(K)\rightarrow C^*(K)$.

(2) The natural morphisms of cochain algebras,
$$C_{PL}(K)\rightarrow(C_{PL}\otimes A_{PL})(K)\leftarrow A_{PL}(K)$$
are quasi-isomorphisms.
\end{Theorem}

\section{main results}\label{33}

Recall that via fixed canonical maps the sphere $S^{n-1}$ and the disk $D^n$ are always identified with $\partial{\Delta^n}$ and $\Delta^n$, respectively. Thus, attaching maps and characteristic maps for $n$-cells of a CW complex $X$, ${f}: S^{n-1}\rightarrow X$,
 $\tilde{f}: D^n\rightarrow X$, can be regarded as singular simplices ${f}: \partial\Delta^n \rightarrow X$ and
  $\tilde{f}: \Delta^n \rightarrow X,$ respectively.

From now on, we assume that the space $X$ is a simply connected CW complex with
a nondegenerate
basepoint, and satisfies the following conditions:

(1) $H^*(X;\mathbb{Q})$ is finite dimensional and vanishes in odd degrees;

(2) The differential $\partial_{\ast}$ in the cellular chain complex $(C_{\ast}(X), \partial_{\ast})$ is zero.
\\
Let $V\subset H^{+}(X;\mathbb{Q})$ be a graded vector space such that
$$H^+(X;\mathbb{Q})=H^+(X;\mathbb{Q})\cdot{H^+(X;\mathbb{Q})}\oplus{V}.$$  The inclusion $V\hookrightarrow {H^+(X;\mathbb{Q})}$ extends to a morphism $\varphi:(\Lambda{V},0)\rightarrow({H^*(X;\mathbb{Q}),0)}$ of cochain algebras. Let $V^{\prime}=\{v_j|~j\in A\}$ be the set of generators of $V$, where $A$ is a totally ordered set.

\begin{Lemma}\label{3.1} With notations as above. Let $E=\{\varphi(v_{t_1}v_{t_2}\cdots v_{t_k})\neq{0}|~v_
	{t_j}\in V^{\prime},~t_j\in A,~t_1\leq t_2\leq\cdots\leq t_k,~k\geq 2,~k\in{\mathbb{Z}}\}$.
If $H^*(X;\mathbb{Q})$ satisfies one of the following conditions:
	\begin{itemize}
		\item[(i)]
		 $ E =\emptyset$;
		 \item[(ii)]
		 $E$ is linearly independent,
	\end{itemize}
	then there exist a cochain algebra $(\Lambda\widetilde{V},d)$ and a CW complex $\widetilde{X}$ which satisfy:
	
(1) $(\Lambda{V},0)$ is a sub cochain algebra of $(\Lambda\widetilde{V},d)$, and $(\widetilde{X},X)$ is a relative CW complex.
		
(2) $\varphi$ can extend to a quasi-isomorphism $\widetilde{\varphi}:(\Lambda\widetilde{V},d)\rightarrow ({H^*(\widetilde{X};\mathbb{Q})}, 0)$.
		
(3) The inclusion $i:X\rightarrow{\widetilde{X}}$ is a rational homotopy equivalence.
\end{Lemma}	
	
\begin{proof}
 Apply the universal coefficient theorem for cohomology,
$$0\longrightarrow{{\rm Ext}_\mathbb{Q}^1(H_*(X;\mathbb{Q}),\mathbb{Q})}\longrightarrow{H^*(X;\mathbb{Q})}\longrightarrow{{\rm Hom}_\mathbb{Q}(H_*(X;\mathbb{Q}),\mathbb{Q})}\longrightarrow{0},$$
we have  	
$H^*(X;\mathbb{Q})\cong{(H_*(X;\mathbb{Q}))^*}$.

Note that the differential $\partial_*$ in the cellular chain complex of $X$ is trivial. Then we have $H_*(X;\mathbb{Q})\cong{(C_*(X)\otimes{\mathbb{Q}},\partial_*\otimes1=0)}$. It follows that
\begin{equation}\label{1}
H^*({X};\mathbb{Q})\cong{(H_*({X};\mathbb{Q}))^*}\cong{(C_*({X})}\otimes{\mathbb{Q})^*}.
\end{equation}
For $(C_*(X),\partial_*)$, by the universal coefficient theorem for cohomology again we can get $$H^*(C^*(C_*({X})\otimes{\mathbb{Q}}))\cong (H_*(C_*({X})\otimes{\mathbb{Q}))^*}=(C_*({X})\otimes\mathbb{Q})^*.$$

Let $V$ be a vector space satisfying
$$H^+(X;\mathbb{Q})=H^+(X;\mathbb{Q})\cdot{H^+(X;\mathbb{Q})}\oplus{V}.$$ The inclusion $V\rightarrow{H^+(X;\mathbb{Q})}$ extends to a morphism $\varphi:(\Lambda{V},0)\rightarrow ({H^*(X;\mathbb{Q})}, 0)$ of cochain algebras. It is obvious that
$H(\varphi;\mathbb{Q})$ is surjective.

We call an element ${v_{i_1}\cdots{v_{i_l}}}\in{\Lambda^{l\geq{2}}{V}}$ a good object if it satisfies the
following conditions:
\begin{itemize}
	\item [(a)] ${v_{i_1}\cdots{v_{i_l}}}\not=0,~i_1\leq\cdots\leq i_l,~\varphi(v_{i_1}\cdots{v_{i_l}})=\varphi(v_{i_1})\cup\cdots\cup\varphi{(v_{i_l})}=0$;
	\item [(b)] for any factor $v_{i_t}\cdots{v_{i_d}}, ~i_1\leq i_t\leq\cdots\leq i_d\leq i_l,~0<d-t<l-1,\\~\varphi(v_{i_t}\cdots{v_{i_d}})\neq{0}$.
\end{itemize}
For each good object $v_{i_1}\cdots{v_{i_l}}$, we add $w_{v_{i_1}\cdots{v_{i_l}}}$ to $V$ and  define the differential $dw_{v_{i_1}\cdots{v_{i_l}}}=v_{i_1}\cdots{v_{i_l}}$. Denote by $(\Lambda\widetilde{V},d)$ the resulting cochain algebra.

For each good object $v_{i_1}\cdots{v_{i_l}}$, set ${\rm deg} (v_{i_1}\cdots{v_{i_l}})=k$. Then we will attach a pair of cells to $X$: a $(k-1)$-cell $\sigma_{v_{i_1}\cdots{v_{i_l}}}$, a $k$-cell $\tau_{v_{i_1}\cdots{v_{i_l}}}$.

Attach the $(k-1)$-cell $\sigma_{v_{i_1}\cdots{v_{i_l}}}$ by the constant map $f_{k-1}:(S^{k-2},s_0)\rightarrow{(x_0,x_0)}$. Denote by $\tilde{f}_{k-1}$ the characteristic map for it, and denote by $\overline{X}$ the temporary resulting space.

Now we construct a attaching map for the $k$-cell $\tau_{v_{i_1}\cdots{v_{i_l}}}$ as follows.
We first choose a map $r: (S^{k-1},s_0)\rightarrow{(S^{k-1},s_0)}$ such that  $H_{k-1}(r)[S^{k-1}]=5[S^{k-1}]\in{H_{k-1}(S^{k-1},s_0)}$, where $[S^{k-1}]$ is the fundamental class of $S^{k-1}$. Then we have the following composite map  $$\Delta^{k-1}\xrightarrow{{\rm projection}}{\Delta^{k-1}/\partial(\Delta^{k-1})}\xrightarrow{\cong}{S^{k-1}}\xrightarrow{r}{S^{k-1}},$$
denoted by $\tilde{r}$.

We now define a map $$g:\partial(\Delta^{k})\rightarrow{{\rm Im}{(\tilde{f}_{k-1}})}$$ as follows.
If $x\in \lambda_i(\Delta^{k-1})=\langle e_0,\cdots,\hat{e}_i,\cdots,e_{k}\rangle,~1\leq{i}\leq{k}$, we let $g(x)=\tilde{r}(x)$. Here $\langle e_0,\cdots,\hat{e}_i,\cdots,e_{k}\rangle$ is identified with $\langle e_0,\cdots,e_{k-1}\rangle$ under the canonical map $e_j\mapsto{e_{j}}$ if $j<{i}$ and $e_j\mapsto{e_{j-1}}$ if $j>i$.
If $x\in \langle e_1,\cdots, e_{k}\rangle$, we let $g(x)=\tilde{f}_{k-1}(x)$. $g$ is well-defined, since
$\tilde{r}(\partial(\Delta^{k-1}))=\tilde{f}_{k-1}(\partial(\Delta^{k-1}))=x_0.$

Attach a new cell $\tau_{v_{i_1}\cdots{v_{i_l}}}$ to $\overline{X}$ with $g$ as an attaching map. The resulting space denoted by $\widetilde{X}$ is our desired CW complex. Denote the differential in its cellular chain complex by $\widetilde\partial_*$.
Note that cells $\sigma_{v_{i_1}\cdots{v_{i_l}}}$ and $\tau_{v_{i_1}\cdots{v_{i_l}}}$  can be regarded as elements in $C_{k-1}(\widetilde{X})$ and $C_k(\widetilde{X})$, respectively.
It is easy to get that $$\widetilde\partial_*(\tau_{v_{i_1}\cdots{v_{i_l}}})=c_{v_{i_1}\cdots{v_{i_l}}}\sigma_{v_{i_1}\cdots{v_{i_l}}},$$ where $c_{v_{i_1}\cdots{v_{i_l}}}$ is a non-zero constant.
Then we have,
 $$\widetilde\partial^*\sigma^*_{v_{i_1}\cdots{v_{i_l}}}=(-1)^{k-1}c_{v_{i_1}\cdots{v_{i_l}}}\tau^*_{v_{i_1}\cdots{v_{i_l}}}.$$

Now define $\widetilde{\varphi}:(\Lambda\widetilde{V},d)\rightarrow{(H^*(\widetilde{X};\mathbb{Q})},0)$ as follows. The following linear map of degree zero from $\widetilde{V}$ to $H^*(\widetilde{X};\mathbb{Q})$
\begin{equation}\nonumber
\begin{cases}

\widetilde{\varphi}(v)=\varphi(v), &{\rm if}\ \ v\in{V};\\

\widetilde{\varphi}(w_{v_{i_1}\cdots{v_{i_l}}})=0,
\end{cases}
\end{equation}
 extends to a unique morphism of graded algebras,
$$\Lambda\widetilde{V}\rightarrow{H^*(\widetilde{X};\mathbb{Q})},$$ denoted by $\widetilde{\varphi}$. Note that $\widetilde{\varphi}(d(v_{i_1}\cdots{v_{i_l}}))=0=d(\widetilde{\varphi}(v_{i_1}\cdots{v_{i_l}}))$ and $\widetilde{\varphi}(dw_{v_{i_1}\cdots{v_{i_l}}})=\widetilde{\varphi}(v_{i_1}\cdots{v_{i_l}})=0=d(\widetilde{\varphi}(w))$. It follows that $\widetilde{\varphi}$ is actually the desired morphism of cochain algebras from $(\Lambda\widetilde{V},d)$ to ${(H^*(\widetilde{X};\mathbb{Q})},0)$.

Finally, we show that $\widetilde{\varphi}:(\Lambda\widetilde{V},d)\rightarrow{(H^*(\widetilde{X};\mathbb{Q}),0)}$ is a quasi-isomorphism and the inclusion $i:X\rightarrow{\widetilde{X}}$ is a rational homotopy equivalence.
$H(\widetilde\varphi)$ is obviously surjective. If $\sum{a_{t_1,\cdots,t_r}v_{t_1}\cdots{v_{t_r}}}\neq 0\in H(\Lambda\widetilde{V},d)$, where
$a_{t_1,\cdots,t_r}\neq 0\in{\mathbb{Q}}$ and $v_{t_1}\cdots{v_{t_r}}\neq 0\in  H(\Lambda\widetilde{V},d)$, combining with condition $(i)$ or $(ii)$, we have
$$H(\widetilde\varphi)(\sum{a_{t_1,\cdots,t_r}v_{t_1}\cdots{v_{t_r}}})=\sum{a_{t_1,\cdots,t_r}\varphi(v_{t_1}\cdots{v_{t_r}}})\neq 0,$$
that is, $H(\widetilde\varphi)$ is injective.
Hence we complete the proof of the isomorphism of $H(\widetilde{\varphi})$.
Note that in the construction of $\widetilde{X}$, we do not change the cellular homology of $X$.
Using the same argument in the beginning of the proof and the structure of $\widetilde{X}$, we can obtain that
\begin{equation}\label{2}
H^*(\widetilde{X};\mathbb{Q})\cong{(H_*(\widetilde{X};\mathbb{Q}))^*}\cong {(H_*(C_*({\widetilde{X}})}\otimes{\mathbb{Q}))^*},
\end{equation}
\begin{equation}\nonumber
H^*(C^*(C_*({\widetilde{X}})\otimes{\mathbb{Q}}))\cong{(H_*(C_*({\widetilde{X}})}\otimes{\mathbb{Q}))^*}.
\end{equation}
 From $(\ref1)-(\ref2)$, the inclusion $i:X\rightarrow{\widetilde{X}}$ is a rational homotopy equivalence.
\end{proof}

\begin{Remark}
 Note that the following isomorphisms
		\begin{align}
		H^*({X};\mathbb{Q})&\cong{(H_*({X};\mathbb{Q}))^*}\cong{(H_*(C_*({X})}\otimes{\mathbb{Q}))^*}\\
		&\cong{(H_*(C_*(\widetilde{X})\otimes{\mathbb{Q}))^*}}\cong{H^*(C^*(C_*({\widetilde{X}})\otimes{\mathbb{Q}}))}\nonumber\\
		&\cong{(H_*(\widetilde{X};\mathbb{Q}))^*}\cong {H^*(\widetilde{X};\mathbb{Q})}.\nonumber
		\end{align}
\end{Remark}

Let $\widetilde\partial_*=\sum\limits_{i}(-1)^id_i$, where $d_i$ are the face maps. For convenience, we denote $d_i(\tau)$ by
$\tau\circ \langle e_0,\cdots,\hat{e}_i,\cdots,e_n\rangle$ for a cell $\tau\in C_*(\widetilde{X})$.

\begin{Lemma}With notations as in Lemma \ref{3.1}.
	There exists an isomorphism of cochain algebras,
	 $$\psi: (\Lambda\widetilde{V}/I,d)\rightarrow{(C^*(C_*({\widetilde{X}})\otimes{\mathbb{Q}}),\widetilde\partial^*),}$$ where $(I,d)\subset(\Lambda\widetilde{V},d)$ is a differential ideal.
\end{Lemma}
\begin{proof}	
For a non-zero element $v_{i_k}\not=0\in{V}$, from the structures of $X$ and $\widetilde{X}$, there exists a non-zero element ${F_{v_{i_k}}}=\sum_{j}a_j z_j^*\in{C^*(C_*({{X}})\otimes{\mathbb{Q}})}$ representing the same cohomology class with $v_{i_k}$, where  $z_j\in{C_*({{X}})\otimes{\mathbb{Q}}}$.

Let $A_{v_{i_k}}=\{\text{good objects}~ v_{i_1}\cdots{v_{i_l}}|~v_{i_k}~\text{is a factor of} ~v_{i_1}\cdots{v_{i_l}}\}$. For each element $v_{i_1}\cdots{v_{i_l}}\in A_{v_{i_k}}$, set ${\rm deg}(v_{i_1}\cdots{v_{i_l}})=k,~{\rm deg}(v_{i_j})=t_j\geq2,~1\leq{j}\leq{l}$, assign to it
 an element  $b_{v_{i_k}}^{v_{i_1}\cdots{v_{i_l}}}=$\\
$\tau_{v_{i_1}\cdots{v_{i_l}}}\circ{\langle\hat{e}_1,\cdots,\hat{e}_{(t_1+\cdots+t_{k-1}-1)},e_{(t_1+\cdots+t_{k-1})},\cdots,e_{(t_1+\cdots+t_{k})},\hat{e}_{(t_1+\cdots+t_{k}+1)},\cdots,\hat{e}_k}\rangle.$ Let $B_{v_{i_k}}=\{b_{v_{i_k}}^{v_{i_1}\cdots{v_{i_l}}}|~v_{i_1}\cdots{v_{i_l}}\in A_{v_{i_k}}\}$,
$C_*({{\widetilde{X}}})\otimes{\mathbb{Q}}=C_*({X})\otimes{\mathbb{Q}}\oplus \mathbb{Q}B_{v_{i_k}}\oplus H$,
 and define $\widetilde F_{v_{i_k}}\in{C^*(C_*({\widetilde{X}})\otimes{\mathbb{Q}})}$ satisfying
 \begin{equation}\nonumber
 \begin{cases}
 \widetilde F_{v_{i_k}}(z)=F_{v_{i_k}}(z),
  &{\rm if}~z\in{C_*({{X}})\otimes{\mathbb{Q}}};\\
 \widetilde F_{v_{i_k}}(b_{v_{i_k}}^{v_{i_1}\cdots{v_{i_l}}})=1, &{\rm if}~b_{v_{i_k}}^{v_{i_1}\cdots{v_{i_l}}}\in B_{v_{i_k}}
;\\
\widetilde F_{v_{i_k}}(z)=0,~ &\text{\rm if}~z\in H.
 \end{cases}
 \end{equation}	
Note that $H^*(m\otimes\mathbb{Q}):H^*(X;\mathbb{Q})\rightarrow H^*(C^*(C_*({{X}})\otimes{\mathbb{Q}}))$
(Theorem \ref{2.1}) is an isomorphism of cochain algebras and the differential $\partial^*$ in $C^*(C_*({{X}})\otimes{\mathbb{Q}})$ is trivial. Then we have that, if $v_{i_1}\cdots{v_{i_l}}$ is a good object, then
 $$F_{v_{i_1}}\cdots{ F_{v_{i_l}}}=(H^*(m\otimes\mathbb{Q}))(\varphi( v_{i_1}\cdots{v_{i_l}}))=0.$$
Set $C_*({{\widetilde{X}}})\otimes{\mathbb{Q}}=C_*({X})\otimes{\mathbb{Q}}\oplus \mathbb{Q}\tau_{v_{i_1}\cdots{v_{i_l}}}\oplus G$. Direct computations show that
\begin{equation}\nonumber
\widetilde F_{v_{i_1}}\cdots{\widetilde F_{v_{i_l}}}(z)=
\begin{cases}
 F_{v_{i_1}}\cdots{ F_{v_{i_l}}}(z)=0, &{\rm if}~z\in C_k({{X}})\otimes{\mathbb{Q}};\\
\epsilon_{v_{i_1}\cdots{v_{i_l}}} a\tau^*_{v_{i_1}\cdots{v_{i_l}}}(z), &{\rm if}~z=a\tau_{v_{i_1}\cdots{v_{i_l}}}\in C_k({{\widetilde{X}}})\otimes{\mathbb{Q}}, ~a\in{\mathbb{Q}},\\&{\rm where}~ \epsilon_{v_{i_1}\cdots{v_{i_l}}}=1 ~or~-1;\\
0, &{\rm if}~z\in G.
\end{cases}
\end{equation}
Hence for a good object $v_{i_1}\cdots{v_{i_l}}$, we have that
 $$\widetilde F_{v_{i_1}}\cdots{\widetilde F_{v_{i_l}}}= \epsilon_{v_{i_1}\cdots{v_{i_l}}}\tau^*_{v_{i_1}\cdots{v_{i_l}}}\in C^*(C_*({\widetilde{X}})\otimes{\mathbb{Q}}).$$

Now we assert that the following three equalities hold.
 \begin{equation}\nonumber
 \begin{cases}
 (\widetilde F_{v_{i_1}}\cdots{\widetilde F_{v_{i_l}}})\cdot\widetilde F_{v_{i_t^{\prime}}}&=0, ~\text{\rm if}~ v_{i_1}\cdots{v_{i_l}} ~\text{\rm is a good object};\\
 \widetilde F_{v_{i_t^{\prime}}}\cdot\sigma^*_{v_{i_1}\cdots{v_{i_l}}}&=0;\\
 \sigma^*_{v_{i_1}\cdots{v_{i_l}}}\cdot\sigma^*_{v_{i_1}^{\prime}\cdots{v_{i_m}^{\prime}}}&=0.
 \end{cases}
 \end{equation}
 It suffices to show that the first equality holds. The other two equalities can be shown similarly.
 Set ${\rm deg} (v_{i_1}\cdots{v_{i_l}})=k, ~{\rm deg}(v_{i_t^{\prime}})=m$, then $(\widetilde F_{v_{i_1}}\cdots{\widetilde F_{v_{i_l}}})\cdot\widetilde F_{v_{i_t^{\prime}}}\in C^{k+m}(C_{k+m}({{\widetilde{X}}})\otimes\mathbb{Q})$.
  If $z\in  C_{k+m}({{\widetilde{X}}})\otimes{\mathbb{Q}},~z=\sum\limits_iz_i^{\prime}+\sum\limits_jz_j^{\prime\prime}$, where $z_i^{\prime}\in C_{k+m}({{{X}}})\otimes{\mathbb{Q}}$ and $z_j^{\prime\prime}=a\sigma_{v_{t_1}\cdots{v_{t_m}}} \text{\rm or } b\tau_{v_{j_1}\cdots{v_{j_r}}}\in C_{k+m}({{\widetilde{X}}})\otimes{\mathbb{Q}},~a,b\in\mathbb{Q}$. Consider
  $(\widetilde F_{v_{i_1}}\cdots{\widetilde F_{v_{i_l}}})\cdot\widetilde F_{v_{i_t^{\prime}}}(z_i^{\prime})$ and $(\widetilde F_{v_{i_1}}\cdots{\widetilde F_{v_{i_l}}})\cdot\widetilde F_{v_{i_t^{\prime}}}(z_j^{\prime\prime})$.
  Note that $\tau_{v_{i_1}\cdots{v_{i_l}}}\cap X=\{x_0\}$ and  $\tau_{v_{i_1}\cdots{v_{i_l}}}\cap\tau_{v_{i_1^{\prime\prime}}\cdots{v_{i_n^{\prime\prime}}}}~(v_{i_1}\cdots{v_{i_l}}\neq v_{i_1^{\prime\prime}}\cdots{v_{i_n^{\prime\prime}}})=\{x_0\}$, it follows that $(\widetilde F_{v_{i_1}}\cdots{\widetilde F_{v_{i_l}}})\cdot\widetilde F_{v_{i_t^{\prime}}}(z_i^{\prime})=0$ and $(\widetilde F_{v_{i_1}}\cdots{\widetilde F_{v_{i_l}}})\cdot\widetilde F_{v_{i_t^{\prime}}}(z_j^{\prime\prime})=0$. Thus the first equality holds.

  Furthermore, from the attaching map for $\tau_{v_{i_1}\cdots{v_{i_l}}},$  it is obvious that  $b_{v_{i_k}}^{v_{i_1}\cdots{v_{i_l}}}\neq \{x_0\}$ and
 $b_{v_{i_k}}^{v_{i_1}\cdots {v_{i_l}}}\neq\tau_{v_{i_1^{\prime\prime}}\cdots{v_{i_m^{\prime\prime}}}}$~($v_{i_1}\cdots{v_{i_l}}$ may equal  $v_{i_1^{\prime\prime}}\cdots {v_{i_m^{\prime\prime}}}$).
Then if $z\in C_*({\widetilde{X}})\otimes{\mathbb{Q}}$,  $\widetilde{\partial}^*(\widetilde{F}_{v_{i_k}})(z)=\widetilde{F}_{v_{i_k}}(\sum a_{v_{i_1}\cdots{v_{i_l}}}\tau_{v_{i_1}\cdots{v_{i_l}}})=0,~a_{v_{i_1}\cdots{v_{i_l}}}\in\mathbb{Q}$, that is, $\widetilde{\partial}^*(\widetilde{F}_{v_{i_k}})=0$.

  Let $(I,d)\subset(\Lambda\widetilde{V},d)$ be a differential ideal generated by $C\cdot\Lambda^{\geq1}\widetilde{V}$ and $D\cdot\Lambda^{\geq1}\widetilde{V}$, where $C$=\{\text{all good objects}\}, $D=\{w_{v_{i_1}\cdots{v_{i_l}}}|~{v_{i_1}\cdots{v_{i_l}}}\in C\}$.
Now we define a linear map of degree zero $\psi: (\Lambda\widetilde{V}/I,d)\rightarrow{(C^*(C_*({\widetilde{X}})\otimes{\mathbb{Q}}),\widetilde\partial^*)}$ as follows:
\begin{equation}\nonumber
\begin{cases}
\psi(v_{i_k})=\widetilde{F}_{v_{i_k}}, &{\rm if}~v_{i_k}\in{V};\\
{\psi}(w_{v_{i_1}\cdots{v_{i_l}}})=(\frac{\epsilon_{v_{i_1}\cdots{v_{i_l}}}}{(-1)^{k-1}c_{v_{i_1}\cdots{v_{i_l}}}}) \sigma^*_{v_{i_1}\cdots{v_{i_l}}},
&~{\rm deg}(v_{i_1}\cdots{v_{i_l}})=k,
\end{cases}
\end{equation}	
and then extend $\psi$ to a morphism of graded algebras. Because $\widetilde\partial^*(\psi(v_{i_1}\cdots{v_{i_l}}))=\widetilde\partial^*(\widetilde F_{v_{i_1}}\cdots{\widetilde F_{v_{i_l}}})=0=\psi(d(v_{i_1}\cdots{v_{i_l}}))$ and $\widetilde\partial^*(\psi(w_{v_{i_1}\cdots{v_{i_l}}}))=\epsilon_{v_{i_1}\cdots{v_{i_l}}}\tau^*_{v_{i_1}\cdots{v_{i_l}}}=\widetilde F_{v_{i_1}}\cdots{\widetilde F_{v_{i_l}}}=\psi(v_{i_1}\cdots{v_{i_l}})=\psi(d(w_{v_{i_1}\cdots{v_{i_l}}}))$,
$\psi$ is essentially  the desired morphism  from $(\Lambda\widetilde{V}/I,d)$ to $(C^*(C_*({\widetilde{X}})\otimes{\mathbb{Q}}),\widetilde\partial^*)$.

Now we show that $\psi$ is injective. Suppose $$c=c_1+c_2\neq 0\in (\Lambda\widetilde{V}/I,d),$$
where $c_1=\sum a_{v_{i_1}\cdots{v_{i_l}}} v_{i_1}\cdots{v_{i_l}},~c_2=\sum b_{v_{j_1}\cdots{v_{j_r}}} w_{v_{j_1}\cdots{v_{j_r}}}$,
  and $a_{v_{i_1}\cdots{v_{i_l}}},~b_{v_{j_1}\cdots{v_{j_r}}}\in\mathbb{Q}$.
  If $c_1=0$, it is obvious that $\psi(c)=\psi(c_2)\neq 0$. Now we consider the case that $c_1\neq 0$. Note that
  \begin{equation}\label{34}
  (H^*(m\otimes\mathbb{Q}))(\varphi(c_1))=\sum a_{v_{i_1}\cdots{v_{i_l}}} F_{v_{i_1}}\cdots{ F_{v_{i_l}}}
  \neq 0.
  \end{equation}
 It follows that $\psi(c_1)=\sum a_{{v_{i_1}\cdots{v_{i_l}}}}\widetilde F_{v_{i_1}}\cdots{\widetilde F_{v_{i_l}}}\neq 0$.
 From $(\ref{34})$,
  there exists an element $z\in C_*({{X}})\otimes{\mathbb{Q}}$ such that
$(\sum  a_{v_{i_1}\cdots{v_{i_l}}}F_{v_{i_1}}\cdots{ F_{v_{i_l}}})(z)\neq 0$.
Then  $\psi(c_1)\neq a\psi(c_2),~a\in\mathbb{Q}$.
Thus $\psi(c)\not=0$, showing that $\psi$ is injective.
 Note that $H^*(X; \mathbb{Q})$ is of finite dimension. From the construction of $\widetilde{X}, ~C_*({\widetilde{X}})\otimes{\mathbb{Q}}$ is of finite type,
then both $\Lambda\widetilde{V}$ and  $C^*(C_*({\widetilde{X}})\otimes{\mathbb{Q}})$ are of finite type.
Thus $\psi$ is an isomorphism.
	\end{proof}

\begin{Lemma}
	$(C^*(C_*({\widetilde{X}})\otimes{\mathbb{Q}}),\widetilde\partial^*)$ is formal.
\end{Lemma}	
	\begin{proof}
		Note that $$ (C^*(C_*({\widetilde{X}})\otimes{\mathbb{Q}}),\widetilde\partial^*)\xrightarrow{\psi^{-1}}{(\Lambda{\widetilde{V}}/I,d)}\xrightarrow{\widetilde{\varphi}}{(H^*(\widetilde{X},\mathbb{Q}), 0)}\cong (H^*(C^*(C_*(\widetilde{X})\otimes{\mathbb{Q}})), 0).$$
		\end{proof}

\begin{Corollary}
	Both $\widetilde{X}$ and $\widetilde{X}_{\mathbb{Q}}$ are formal.
\end{Corollary}
	\begin{proof}
	From Theorem \ref{2.5}, we have a sequence of quasi-isomorphisms of cochain algebras,
		\begin{equation}
		{(C^*(C_*({\widetilde{X}})\otimes{\mathbb{Q}}),\widetilde\partial^*)}\rightarrow \bullet\leftarrow (A_{PL}(C_*({\widetilde{X}})\otimes{\mathbb{Q}}), d)
		\xleftarrow{A_{PL}(m)}{(A_{PL}(S_*(\widetilde{X};\mathbb{Q})), d)}.
		\end{equation}
		Note that $(C^*(C_*({\widetilde{X}})\otimes{\mathbb{Q}}),\widetilde\partial^*)$ is formal. It follows that ${(A_{PL}(S_*(\widetilde{X};\mathbb{Q})), d)}$ is also formal, showing that $\widetilde{X}$ is formal.
		Recall that the inclusion $\widetilde{X}\rightarrow \widetilde{X}_{\mathbb{Q}}$ induces a rational homotopy equivalence. It follows that $\widetilde{X}_{\mathbb{Q}}$ is formal.
		\end{proof}

We now prove Theorem \ref{1.1}.\\
\begin{proof}[Proof of Theorem \ref{1.1}]
		For the CW complex $Y$, we can make use of Theorem \ref{2.2} to obtain a CW complex $X$ satisfying the conditions of Lemma \ref{3.1}. Then we have the following commutative diagram,
		$$
		\xymatrix{
			& \widetilde{X}\ar[d] &           	X \ar[l]\ar[d]^{i_X}\ar[r]^{f}   &   Y\ar[d]_{i_Y}  \\ &\widetilde{X}_{\mathbb{Q}}   &X_{\mathbb{Q}}\ar[l]\ar[r]^{f_{\mathbb{Q}}} &  Y_{\mathbb{Q}},}
		$$
		where $i_X$ and $i_Y$ are inclusions.
		Since the first horizontal row is a chain of quasi-isomorphisms, CW complexes $\widetilde{X}$ and $Y$ have the same rational homotopy type. Then we apply Theorem \ref{2.3} to obtain $$\widetilde{X}_{\mathbb{Q}}\simeq{Y_{\mathbb{Q}}}.$$ Since $\widetilde{X}_{\mathbb{Q}}$ is formal, $Y_{\mathbb{Q}}$ is formal.
		\end{proof}

\section{ An example}\label {44}
Denote by $\langle~\rangle$ the
Sullivan's realization functor \rm{\cite{Sullivan}} from commutative cochain algebras to
simplicial sets, and denote by $|\quad~|$ the Milnor's realization functor \rm{\cite{Milnor}}  from simplicial sets to CW complexes. Then the spatial realization of a commutative cochain algebra $(A,d)$ is the CW complex $|\langle A, d  \rangle|$ (abbreviated as $|A, d|$ thereafter).

\begin{Example}
	With notations as in \cite[Theorem 17.10]{06bumu}.	Suppose $(\Lambda V, d)$ is a  Sullivan algebra with finite dimensional rational cohomology. Let $F=\{n_j\in\mathbb{Z}|~0<n_1<\cdots<n_j<\cdots,~ H^{n_j}(\Lambda V, d)\neq 0\}$.
	If $H(\Lambda V, d)$ vanishes in odd degrees and satisfies one of the following conditions:
	\begin{itemize}
		\item [(i)] F=$\emptyset$;
		\item [(ii)]
		for any $n_k\in F$,
		$n_k\neq\sum\limits_{1\leq i<k}a_in_i,~a_i\in\mathbb{Z},$
	\end{itemize}
	then $|\Lambda V,d|_{\mathbb{Q}}$ is formal. In particular, if $(\Lambda V, d)$ is a  Sullivan algebra with at most two non-trivial cohomology groups $H^0(\Lambda V, d)$ and $H^{2n}(\Lambda V, d),~n>0$,  $|\Lambda V,d|_{\mathbb{Q}}$ is formal.
\end{Example}
\begin{proof}	
From \cite[Theorem 17.10]{06bumu}, $|\Lambda V,d|$ is simply connected,  $\zeta_n: \pi_n(|\Lambda V,d|) \xrightarrow {\cong}{\rm Hom}(V, \mathbb{Q})$ is an isomorphism and $ m_{( \Lambda V,d)}: ( \Lambda V,d) \xrightarrow {\simeq} A_{PL}(|\Lambda V,d|)$ is a quasi-isomorphism. Then $|\Lambda V,d|$ is a rational space and $H(A_{PL}(|\Lambda V,d|)) {\cong} H( \Lambda V,d)$. Note that $|\Lambda V,d|_{\mathbb{Q}}{=}|\Lambda V,d|$. It follows that
$$H(|\Lambda V,d|_{\mathbb{Q}}) {=} H(|\Lambda V,d|)  {\cong} H(A_{PL}(|\Lambda V,d|)) {\cong} H( \Lambda V,d).$$ According to the hypothesis, $|\Lambda V,d|_{\mathbb{Q}}$ satisfies the conditions of Corollary \ref{1.2}, implying that $|\Lambda V,d|_{\mathbb{Q}} {=} |\Lambda V,d|$ is formal.
\end{proof}

\end{document}